\documentclass{amsart}
\usepackage{custom}

\usepackage{graphicx, mathrsfs, subfigure}
\usepackage{psfrag,epsf}

\theoremstyle{plain}
\newtheorem*{lawsonsfirst}{Lawson's First Theorem}
\newtheorem*{lawsonssecond}{Lawson's Second Theorem}

\newtheorem*{radons}{Radon's Theorem}
\newtheorem*{caratheodorys}{Carath\'eodory's Theorem}

\theoremstyle{definition}
\newcommand\T{\mathscr{T}}

\newcommand\Vol{\qopname\relax o{Vol}}
\newcommand\conv{\qopname\relax o{conv}}
\newcommand\origin{\mathbf{0}}

\newcommand\matrixm{\mathbf{M}}
\newcommand\matrixa{\mathbf{A}}
\newcommand\matrixi{\mathbf{I}}
\renewcommand\int{\qopname\relax o{int}}
\newcommand\aff{\qopname\relax o{aff}}
\renewcommand\P{\mathscr{P}}
\newcommand\Skd{\mathscr{S}_k^d}
\newcommand\V{\mathscr{V}}
\renewcommand\A{\mathscr{A}}
\renewcommand\S{\mathscr{S}}

\begin{document}

\title{Minimal Volume $k$-point Lattice $d$-Simplices}
\author{Han Duong}
\begin{abstract}
We show via triangulations that for $d\ge 3$ there is exactly one class (under
unimodular equivalence)
of nondegenerate lattice
simplices in $\R^d$ with minimal volume
and $k$ interior lattice points.
\end{abstract}
\date{\today}
\maketitle

%%%%%%%%%%%%%%%%%%%%%%%%%%%%%%%%%%%%%%%%%%%%%%%%%%%%%%%%%%%%%%%%%%%%%%%%%%%%
%
% INTRODUCTION
%
%%%%%%%%%%%%%%%%%%%%%%%%%%%%%%%%%%%%%%%%%%%%%%%%%%%%%%%%%%%%%%%%%%%%%%%%%%%%
\section{Introduction}

A {\em $d$-polytope} $P$ is a polytope of dimension $d$.
If its {\em vertex set} $\V(P)$ is a subset of $\Z^d$, then
$P$ is a {\em lattice} $d$-polytope. If in
addition $|\V(P)| = d+1$, then $P$ is a
lattice {\em $d$-simplex}.
The convex hull of
$\P = \{ v_1, \dotsm, v_n \} \subset \Z^d$, denoted by $\conv(\P)$,
is a lattice polytope with at most $n$ vertices
and dimension at most $d$. This notation will be used loosely;
for convenience, we use 
$\conv(P, v)$ to mean $\conv(\V(P) \cup \{ v \})$
when it is clear $P$ is a polytope and $v$ is a point.
As used in \cite{BR}, we say
that $P$ is {\em clean} if $\partial P \cap Z^d = \V(P)$,
where $\partial P$ is the boundary of $P$. If in addition $\int(P)$, 
the {\em interior} of $P$, contains  $k$  lattice points, then $P$ is a clean
{\em $k$-point} lattice polytope. If $k=0$, then the polytope is {\em empty}.
We use $\Skd $ to denote the collection of clean $k$-point lattice
$d$-simplices. Unless otherwise stated,
all polytopes are taken to be convex $d$-polytopes.

Reznick proved in \cite{BR86} and \cite{BR} that any lattice tetrahedron with at least one
clean face is unimodularly equivalent to some $T_{a,b,n}$, the 
lattice tetrahedron with vertex set 
\[ \{\ (0,0,0),\ (1,0,0),\ (0,1,0),\ (a,b,n)\ \} \]
where $(a,b,n)\in \Z^3$ and $0<a,b<n$.
Reznick also classified the set of clean 1-point tetrahedra,
up to equivalence under unimodular transformations,
using barycentric coordinates.  Very recently, Bey, Henk, and Wills
proved in \cite{CBMHJW} that if $P$ is a lattice $d$-polytope, not necessarily clean,
and $P$
has $k$ interior lattice points, then for $d\ge 1$, the volume of $P$ satisifies
\begin{equation}\label{volumeboundeq}
\Vol(P) \ge \frac{1}{d!}(dk+1).
\end{equation}
Moreover, they showed that for $k=1$, equality holds if and only if $P$
is unimodularly equivalent to the simplex $S_d(1)$, where
\[ S_d(k) = 
\conv \left(e_1, \dotsm, e_d, -k \sum_{i=1}^d e_i\right) \]
and $e_i$ denotes the $i$-th unit point. 
This is not true for $d=2$ and $k>2$.
We will show that equality holds in (\ref{volumeboundeq}) for all $k>0$
if and only if $d\ge 3$ and $P$ is unimodularly equivalent to $S_d(k)$.
We first prove that if $T\in \Skd $ and $\Vol(T) = \frac{1}{d!}(d k+1)$, 
then the interior points lie on a line passing through some vertex of $T$.
We then show that such simplices are unimodularly equivalent to 
$T_{a_1, \dotsm, a_{d}}$,
the $d$-simplex whose vertex set consists of the origin, the
points $e_i$ $(1\le i \le d-1)$, and the point
$(a_1, \dotsm, a_{d})$, where $a_j = dk$ for $1\le j < d$ and $a_{d} = dk+1$.
Finally, we will show that $S_d(k)\in \Skd$ and
$\Vol(S_d(k)) = \frac{1}{d!}(dk+1)$.

%%%%%%%%%%%%%%%%%%%%%%%%%%%%%%%%%%%%%%%%%%%%%%%%%%%%%%%%%%%%%%%%%%%%%%%%%%%%
%
% PRELIMINARIES
%
%%%%%%%%%%%%%%%%%%%%%%%%%%%%%%%%%%%%%%%%%%%%%%%%%%%%%%%%%%%%%%%%%%%%%%%%%%%%
\section{Preliminaries}

The following definitions are taken from
\cite{MBSR}, \cite{BG}, \cite{BBP}, and \cite{GZ}.
A {\em $j$-flat} is a $j$-dimensional affine subspace of $\R^d$.
{\em Points}, {\em lines}, and {\em planes} are $0$-flats,
$1$-flats, and $2$-flats, respectively. The 
{\em affine hull} of a set $\P\subset \R^d$,
denoted by $\aff(\P)$,
is the intersection of all flats containing $\P$. 
Equivalently, $\aff(\P)$ is the smallest flat containing $\P$. We say
$\P$ is in {\em general position} if no $j+2$ points of $\P$
lie in a $j$-flat, where $j<d$.
A hyperplane
\[ H = \{ x \in \R^d : \mathbf{a\cdot x} = b \} \]
is a $(d-1)$-flat. 
If $P$ is a $d$-polytope, then
$H$ is a {\em supporting hyperplane} of 
$P$ if $P$ lies entirely on one side of $H$.
A {\em face} of $P$ is an intersection
$P \cap H$, where $H$ is a supporting hyperplane. 
If we allow degenerate hyperplanes, then $P$ is a face of $P$
corresponding to $H=\R^d$; $\emptyset$ is also a face of $P$ corresponding
to a hyperplane that does not meet $P$.
A {\em $j$-face} of $P$ is a $j$-dimensional face of $P$.
A $(d-1)$-face is a {\em facet}, a $1$-face is an {\em edge},
and a $0$-face is a {\em vertex}.
One property of $d$-polytopes is that any $j$-face of $P$ is
contained in at least $d-j$ facets of $P$.

\begin{lemma} \label{ifaceinjface} \cite[Section 3.1]{BG}
If $0\le i < j \le d-1$ and if $P$ is a $d$-polytope, each $i$-face of
$P$ is the intersection of the family 
of $j$-faces of $P$ containing it. There are at least $j+1-i$ such faces.
\end{lemma}

We generally use capital letters to denote $d$-polytopes.  In particular
$P$ and $Q$ are  $d$-polytopes, and $S$ and $T$ are $d$-simplices.
Capital script letters will generally denote sets of $d$-simplices.
In particular, $\P$ is a set of points ($0$-simplices).

%%%%%%%%%%%%%%%%%%%%%%%%%%%%%%%%%%%%%%%%%%%%%%%%%%%%%%%%%%%%%%%%%%%%%%%%%%%%
%
% TRIANGULATIONS AND REFINEMENTS
%
%%%%%%%%%%%%%%%%%%%%%%%%%%%%%%%%%%%%%%%%%%%%%%%%%%%%%%%%%%%%%%%%%%%%%%%%%%%%
\section{Triangulations and Refinements}

Borrowing from \cite{CLL}, we
let $\P$ denote a set of $n$ distinct points in $\R^d$,
where $n \ge d+1$ and $d\ge 2$.
Assume $\P$ does not lie entirely in a hyperplane. Let $P=\conv(\P)$. A
{\em triangulation}, $\T$, of $\P$ (or of $P$ with the dependence on $\P$
understood) is a set of nondegenerate $d$-simplices $\{ T_i \}$
with the following properties.
\renewcommand{\labelenumi}{(\alph{enumi})}
\begin{enumerate}
\item All vertices of each simplex are members of $\P$.
\item The interiors of the simplices are pairwise disjoint.
\item Each facet of a simplex is either on the boundary of $P$,
or else is a common facet of exactly two simplices.
\item \label{partialcondition} Each simplex
contains no points of $\P$ other than its vertices.
\item The union of $\{ T_i \}$ 
 is $\P$ and the union of $T_i$ is $P$.
\end{enumerate}
\renewcommand{\labelenumi}{(\arabic{enumi})}
Since each $d$-simplex has volume at least $\frac{1}{d!}$, one
immediate consequence is
\begin{equation}\label{volbytriangulation}
\Vol(P) \ge \frac{1}{d!} \cdot |\T|.
\end{equation}
To prove (\ref{volumeboundeq}), Bey, Henk, and Wills showed $P$
with $k$ interior lattice points can be
decomposed into at least $dk+1$ nondegenerate $d$-subpolytopes.
Any $d$-polytope must contain a $d$-simplex as a subpolytope, so
(\ref{volbytriangulation})
still holds if $\T$ is replaced by this decomposition. We
will present a slight variation of their theorem and its proof by
using triangulations.

\begin{dfn} Let $P$ be a lattice $d$-polytope.
A {\em lattice triangulation} $\T$ of $P$
is a triangulation of some set $\P\subset \Z^d$
such that $\V(P) \subseteq \P \subseteq P\cap \Z^d$.
Note that if $\V(P) \subseteq \P, \P' \subseteq P\cap \Z^d$
and $\P \not= \P'$, we still have
$\conv(\P) = P = \conv(\P')$.
On the other hand, 
the triangulation $\T$ of $\P$ and
the triangulation $\T'$ of $\P'$ are necessarily different.
Fortunately, we can reconstruct the 
{\em vertex set} of $\T$, denoted by  $\V(\T)$,
by taking the union of all vertices of all $T\in \T$. Thus
$\V(\T) = \P$ and $\V(\T') = \P'$.
\end{dfn}

\begin{dfn}
Let $P$ be a lattice polytope
and $\T$ be a lattice triangulation $P$. We say
$\T'$ is a {\em refinement} 
of $\T$ 
(and write $\T \prec \T'$)
provided $\T'$ is a lattice
triangulation of $P$,
$\V(\T) \subsetneq \V(\T')$, and for all $T'\in \T'$
there exists  $T\in \T$ such that $T'\subseteq T$.
We say that $\T$ is a {\em full lattice triangulation}
if $\V(\T) = P\cap \Z^d$. Otherwise we say $\T$ is a {\em
partial lattice triangulation}. 
\end{dfn}

Naturally, triangulations of $\P\subset \R^d$
partition $\conv(\P)$ into simplices.
On the other hand, there exist partitions of $\conv(\P)$ that satisify all
but condition (d) in the definition of a triangulation.

\begin{figure}[h] 
\centering
%\psfrag{v1}{\smaller $(0,1)$}
%\psfrag{v2}{\smaller $(1,0)$}
%\psfrag{v3}{\smaller $(0,-1)$}
%\psfrag{v4}{\smaller $(-1,0)$}
\psfrag{v1}{\smaller $e_2$}
\psfrag{v2}{\smaller $e_1$}
\psfrag{v3}{\smaller $-e_2$}
\psfrag{v4}{\smaller $-e_1$}
\psfrag{v5}{\smaller $\mathbf{0}$}
\psfrag{T1}{\smaller $\mathscr{Q}_1$}
\psfrag{T2}{\smaller $\mathscr{Q}_2$}
\psfrag{T3}{\smaller $\mathscr{Q}_3$}
\includegraphics[scale=1]{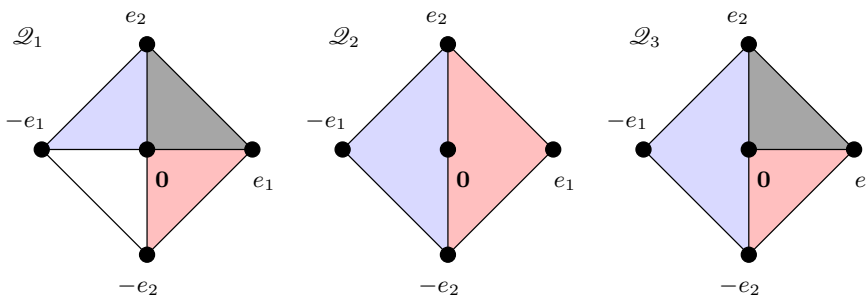}
\caption{Partitions of 
$P$ into triangles\label{triangulationexamples}}
\end{figure}

\begin{example}
Let $\P = \{ e_1, e_2, -e_1, -e_2 \} \subset \R^2$
and let $P = \conv(\P)$.
Note that $P$ is a 1-point lattice polygon, and its
interior lattice point is the origin.
Three possible partitions $\mathscr{Q}_1$,
$\mathscr{Q}_2$, and $\mathscr{Q}_3$ of $P$ into 
triangles are shown in Figure \ref{triangulationexamples}.
Note that $\mathscr{Q}_1$ is a triangulation of $\P\cup \{ (0,0)\}$
and a full lattice triangulation of $P$. However, $\mathscr{Q}_1$
 is not a triangulation of $\P$.
The middle partition $\mathscr{Q}_2$ is
both a triangulation of $\P$ and
a partial lattice triangulation of $P$ (viewed as a triangulation of $\V(P)=\P$).
The partition $\mathscr{Q}_3$ on the right
is not a triangulation since it fails condition
(d).
\end{example}

The following theorem,
the proof of which can be found in the appendix of \cite{MBSR}, guarantees
the existence of a triangulation of the vertex set of
a polytope.

\begin{theorem} \label{nonewvert}
\cite[Theorem 3.1]{MBSR}
Every convex polytope $P$ can be triangulated using no new vertices. That is,
there exists a triangulation of $\V(P)$.
\end{theorem}

Let $P$ be a clean,
non-empty lattice $d$-polytope, and suppose $w\in \int(P)\cap \Z^d$.
We construct a {\em basic lattice triangulation} $\T_w$
of $P$ in the following 
manner.
Theorem \ref{nonewvert} guarantees that
each facet $F$ of $P$, as a $(d-1)$-polytope, 
has a (lattice) triangulation $\T_F$ of $\V(F)$. Let $\mathscr{F}$
be the set of facets of $P$ and let $\mathscr{B}$ be the set
\[ \mathscr{B} = \bigcup_{F\in \mathscr{F}} \T_F \]
of $(d-1)$-simplices. Finally, let 
$\T_w =  \{\ \conv(S,w)\ :\ S\in \mathscr{B}\ \}$.
It is easy to check that $\T_w$
is a lattice triangulation of $P$. Note that
\begin{equation} \label{sizeofbasictriangulations}
 |\T_w| = |\mathscr{B}| \ge |\mathscr{F}| \ge d+1.
\end{equation}

In particular, if 
$T\in \Skd$, where $k\ge 1$, and  $F_1,\dotsm,F_{d+1}$ are the
facets of $T$, then a
basic lattice triangulation $\T_w$ is simply the convex
hull of the facets of $T$ with an interior lattice point $w$ of $T$.
Moreover, $\T_w$ is a refinement of the trivial triangulation
$\T_0 = \{ T \}$.
The main idea in proving the collinearity property 
of $\int(T)\cap \Z^d$ is to start with
the trivial triangulation $\T_0 = \{ T \}$
and obtain a sequence 
$\T_1  \prec \dotsm \prec \T_k$
of refinements 
such that $\T_k$ is a full triangulation of $T$ and
$|\T_k|\ge dk+1$, and then show that noncollinearity
forces $|\T_k|>dk+1$. 
The following lemma appears as an assertion in
the proof of (\ref{volumeboundeq}) in \cite{CBMHJW}. We prove it here
since it is crucial in computing $|\T_i| - |\T_{i-1}|$.

\begin{lemma} \label{atleastdplusoneminusj}
Suppose $P$ is a clean, non-empty, lattice $d$-polytope.
Let $\T$ be a partial triangulation of $P$.
If $S$ is a $j$-face of some $T\in \T$, and
$\V(S)\cap \int(P)\cap \Z^d \not= \emptyset$,
then $S$ is contained in at least $d+1-j$ simplices in $\T$.
\end{lemma}

\begin{proof} Let $w\in \V(S)\cap \int(P)\cap \Z^d$. 
By Lemma \ref{ifaceinjface}, $S$ is contained in at least $d-j$ facets of $T$.
Moreover $w$ is a vertex of each facet of $T$ containing $S$. It follows
that these facets are not contained in facets of $P$ and are therefore
shared by exactly two simplices in $\T$.
On the other hand any two simplices in $\T$ intersect in at
most one common facet. Thus  $S$ is contained in at least
$d-j$ other simplices in $\T$.
\end{proof}

\begin{theorem} \label{refinementtheorem}
Let $P$ be a clean, non-empty, lattice $d$-polytope.
Let $\T$ be a partial triangulation of $P$. 
For any $w\in \int(P)\cap \Z^d \backslash \V(\T)$
there exists  a refinement $\T'$ of $\T$ such that
$\V(\T')=\V(\T)\cup \{ w \}$ and $|\T'| \ge |\T| + d$.
\end{theorem}

\begin{proof}
Since $\T$ is a partial triangulation, 
there exists an interior lattice point $w$ of $P$
such that $w\not\in \V(\T)$.
Moreover, $w$ must lie in the relative interior
of some  $j$-face ($1 \le j \le d$), say $S$, 
of some simplex in $\T$.
Note that $S$ is in fact a $j$-simplex.
Let $\V(S) = \{ v_1,\dotsm, v_{j+1} \}$ and consider the basic
triangulation $\T_w$ of $S$ into $j$-simplices, where
\[ \T_w = \{\  \conv(\V(S)\cup \{ w \} \backslash \{ v_i \})
\ :\ 1 \le i \le j+1\ \}. \]
Clearly $\T_w$ is a refinement of $S$ into $j$-simplices.
This refinement of $S$ induces a refinement of any $d$-simplex containing $S$.
More precisely, if $T\in \T$ contains $S$, then 
the set 
\[ \{\ \conv(\V(\T) \cup \{w\} \backslash \{v_i\} )\ \} \]
is a lattice triangulation of $T$.
Since $w\in \int(S)$ and $P$ is clean, $S\not\subset \partial P$.
Thus $\V(S)\cap \int(P)\cap \Z^d \not=\emptyset$.
By Lemma \ref{atleastdplusoneminusj}, there are at least
$d+1-j$ simplices in $\T$ containing $S$ as a $j$-face.
Now consider $\T$ with all such $d$-simplices in $\T$ 
replaced with their respective induced triangulations and
take this to be $\T'$.  By construction, $w$ is contained in
a simplex  $T'\in \T'$ if and  only if $w\in \V(T')$.
Hence $\T'$ is a refinement of $\T$ such that
$\V(\T') = \V(T) \cup \{ w \}$ and
\[ |\T'| \ge |\T| + (d+1-j)(j+1) - (d+1-j) = |\T|+ (d+1-j)j. \]
Finally,
\begin{equation} \label{atleastd}
(d+1-j)j - d = (d-j)(j-1) \ge 0 \quad\text{for}\quad 1\le j\le d,
\end{equation}
which implies  $(d+1-j)j \ge d$.
\end{proof}

In the proof above,
it is important to note that equality in (\ref{atleastd}) holds
if and only if $j=d$ or $j=1$.
Equally important is that
if $\V(\T') \not= P\cap \Z^d$, then $\T'$ is again a partial
triangulation.

\begin{corollary} \label{simplextriangulation}
If $T\in \Skd$ and $\T_0 = \{ T \}$, then there exists a sequence
\begin{equation} \label{refinementsequence}
\T_0 \prec \dotsm \prec \T_k
\end{equation}
of refinements of $\T_0$
such that $\T_k$ is a full triangulation of $T$ and
$|\T_k| \ge dk+1$. Moreover, $\Vol(T) \ge \frac{1}{d!}(dk+1)$.
\end{corollary}

\begin{proof}
For $k=0$, then the sequence consists only of $\T_0$.
If $k>0$ then 
$\T_0 = \{ T \}$ is a partial triangulation of $T$.
Let $w_1,\dotsm, w_k$ be an arbitrary enumeration
of the interior lattice
points of $T$. For $1\le i \le k$,
we refine $\T_{i-1}$ into $\T_{i}$, using
Theorem \ref{refinementtheorem} with $w=w_i$. 
After $k$ refinements, 
$\T_k$ is a full triangulation of $T$, and
$|\T_k| \ge |\T_0| + dk = dk+1$. 
Lastly, (\ref{volbytriangulation}) implies
$\Vol(T)\ge \frac{1}{d!}(dk+1)$.
\end{proof}

\begin{corollary} \label{nojsimplex}
Suppose $T\in \Skd $ where $k\ge 2$. Let
$w_1,\dotsm, w_k$ be an arbitrary enumeration of
$\int(T)\cap \Z^d$, and let (\ref{refinementsequence})
be the corresponding refinement sequence guaranteed by Corollary
\ref{simplextriangulation}.
If $w_{i+1}$ lies in a $j$-face of a simplex in $\T_i$, 
where $i>0$ and $1<j<d$, then
$|\T_{i+1}|-|\T_i|>d$ and
$\Vol(T) > \frac{1}{d!}(dk+1)$.
\end{corollary}

\begin{proof}
This follows immediately from Theorem \ref{refinementtheorem} and
(\ref{sizeofbasictriangulations}), and (\ref{atleastd}).
\end{proof}

In the context of Corollary \ref{simplextriangulation},
(\ref{atleastd}) implies that for
each $\T_i$ in (\ref{refinementsequence}),
$|\T_i| \ge d i+1$.
Since equality in (\ref{atleastd}) holds if and only if $j=1$
or $j=d$,   $|\T_i| = d i+1$ if and only if for each $i$,
 $w_{i+1}$ lies in
the relative interior of an edge in $\T_i$, or in the relative
interior of a simplex in $\T_i$. Note that this must be true
for any ordering of the $w_i$'s.

Finally, to prove (\ref{volumeboundeq}) for a general
non-empty lattice $d$-polytope $P$,
we start with a basic triangulation of $P$ and obtain a refinement sequence
similar to that of Corollary \ref{simplextriangulation}.

\begin{corollary} \label{equivproof}
\cite[Theorem 1.2]{CBMHJW} 
If $P$ is a lattice $d$-polytope with $k\ge 0$ interior
lattice points, then there exists a sequence
\[\T_1 \prec \dotsm \prec \T_k \]
of lattice triangulations of $P$ such that $\T_1$
is a basic triangulation of $P$, $\T_k$ is a full triangulation
of $P$ and $|\T_k|\ge dk+1$. Moreover, $\Vol(P) \ge \frac{1}{d!}(dk+1)$.
\end{corollary}

\begin{proof}
If $k=0$, take $\T_1$ to be the triangulation guaranteed by Theorem
\ref{nonewvert}. If $k\ge 1$, then consider an arbitrary
enumeration $w_1,\dotsm,w_k$ of the the interior points
of $P$. Let $\T_1$ be the basic triangulation $\T_{w_1}$ of $T$.
Either $k=1$ and $\T_1$ is a full triangulation, or
we can apply Theorem \ref{refinementtheorem}, as
in the proof of Corollary \ref{simplextriangulation}, to
obtain a refinement sequence
\[ \T_1 \prec \dotsm \prec \T_k \]
such that $|\T_k| \ge dk+1$. Finally, (\ref{volbytriangulation})
implies $\Vol(P)\ge \frac{1}{d!}(dk+1)$.
\end{proof}

In general, we
need not start with a basic triangulation of $P$.
It is easy to check that so long
as $\T$ is a partial triangulation 
such that 
\[
 |P\cap \Z^d| - |\V(\T)| = k-j \quad \text{and} \quad |\T| \ge dj+1,
\]
then we can still refine $\T$ into a full triangulation with
at least $dk+1$ simplices by applying Theorem \ref{refinementtheorem}.
This is in fact equivalent to the inductive step in \cite{CBMHJW},
with triangulations replaced by decompositions into $d$-subpolytopes.
The proof of Corollary \ref{equivproof} is otherwise essentially the same
as that in \cite{CBMHJW}.

%%%%%%%%%%%%%%%%%%%%%%%%%%%%%%%%%%%%%%%%%%%%%%%%%%%%%%%%%%%%%%%%%%%%%%%%%%%%
%
% PROPERTIES OF d+2 POINTS
%
%%%%%%%%%%%%%%%%%%%%%%%%%%%%%%%%%%%%%%%%%%%%%%%%%%%%%%%%%%%%%%%%%%%%%%%%%%%%

\section{Properties of $d+2$ Points in $\R^d$}

In the context of Corollary \ref{nojsimplex}, if
we can show $j=d$ implies $\Vol(T) > \frac{1}{d!}(dk+1)$ as well,
then any two interior lattice points of $T$ must be collinear
with some vertex of $T$. 
As a base case, we first
consider the possible configurations of
any $T\in \S_2^d$, where $d\ge 3$.
Suppose $T$ has interior
lattice points $w_1$ and $w_2$,  and let $\V(T) = \{ v_1 ,\dotsm, v_{d+1} \}$.
The sets 
\begin{equation} \label{pisets}
\P_i = \V(T)\backslash \{ v_i \} \cup \{ w_1, w_2 \}
\end{equation}
are all sets of $d+2$ points not contained  in a hyperplane.
Many  properties of such sets are discussed in 
\cite{MB}, \cite{WHJK}, \cite{LK}, \cite{CLL}, \cite{BBP},
\cite{IVP}, and \cite{JR}.
Two relevant and well known results in the theory of convex bodies
are Carath\'eodory's theorem \cite{CC} and Radon's theorem \cite{JR}.

\begin{caratheodorys} 
If $\P = \{ v_1,\dotsm, v_{d+1} \}
\subset \R^d$  is not contained in a hyperplane,
then every $x\in \R^d$ 
can be expressed as
\[ x = \sum_{i=1}^{d+1} \alpha_i v_i,
\quad\text{where}\quad
\alpha_i \in \R, 
\quad v_i \in \P,
\quad \text{and} \quad
\sum_{i=1}^{d+1} \alpha_i = 1. \]
\end{caratheodorys}

The coefficients $\alpha_i$ in Carath\'eodory's theorem
are the {\em barycentric coordinates} of $x$
relative to $\conv(\P)$.
If   $P=\conv(\P)$ is a $d$-simplex, then the barycentric
coordinates of $x$ relative to the simplex $P$ are
the numbers $\alpha_1,\dotsm,\alpha_{d+1}$ satisfying
\[
\underbrace{
\left[ \begin{array}{cccc}
\alpha_1 & \alpha_2 & \dotsm & \alpha_{d+1} 
\end{array} \right]}_{1\times (d+1)}
\cdot 
\underbrace{
\left[ \begin{array}{cc}
v_1 & 1 \\
v_2 & 1 \\
\vdots  & \vdots \\
v_{d+1} & 1 \\
\end{array}
\right]}_{(d+1)\times (d+1)}
= 
\underbrace{\left[ \begin{array}{cc}
x & 1 \end{array} \right]}_{1\times (d+1)} .
\]
For each $i$, the sign of $\alpha_i$ indicates the position
of $x$ relative to the hyperplane $H_i$ containing
the facet of $P$ opposite vertex $v_i$. That is, $\alpha_i>0$
when $v_i$ and $x$ are on the same side of $H_i$,
$\alpha_i < 0$ if $v_i$ and $x$ are on opposite sides of $H_i$,
and $\alpha_i = 0$ if $x$ lies in $H_i$. The barycentric
coordinates of $x$ relative to $T$ are all positive if and only if
$x\in \int(P)$.
Thus any point $x\in \P_i$ (cf. (\ref{pisets})), 
can be described  in terms 
of its barycentric coordinates relative to the simplex $\conv(\P_i)$.

\begin{radons}
If $\A$ is a set of $k\ge d+2$ points in $\R^d$, then there exist
a partition $\{ \A_1, \A_2 \}$ of $\A$
such that $\A = \A_1 \cup \A_2$
and $\conv(\A_1) \cap \conv(\A_2) \not=\emptyset$.
\end{radons}

The partition in Radon's theorem is called
a {\em Radon partition} of $\A$. A Radon partition
 {\em in} $\A$ is a Radon partition of a subset of $\A$.
Let  $\{\A_1, \A_2 \}$ and  $\{\A_1', \A_2' \}$ be Radon partitions
of $\A$.
Then  $\{\A_1, \A_2 \}$ {\em extends} 
$\{ \A_1', \A_2' \}$ provided $\A_i' \subseteq \A_i$.
In \cite{WHJK},
Hare and Kennely introduced the notion of a
{\em primitive Radon partition}, a Radon partition that
is minimal with respect to extension.
An immediate consequence is that 
if $\{ \A_1, \A_2 \}$ is a Radon partition of $\A$,
then there exists a primitive Radon partition in $\A$ such that
$\{ \A_1, \A_2 \}$ extends it.
Breen proved in
\cite{MB} that  $\{ \A_1, \A_2\}$ is
a primitive Radon partition in $\A$ if and only if $\A_1\cup \A_2$
is in general position in $\R^{|\A_1|+|\A_2|}$. Recall that
$\A_1 \cup \A_2$ is in general position if no $j+2$ points in this union
lie in a $j$-flat for all $j<|\A_1|+|\A_2|$.
Peterson proved in \cite{BBP} that
the Radon partition of $d+2$ points in general position in $\R^d$
is unique, and both Breen and Peterson showed that if
$\A_1\cup \A_2$ is in general position in $\R^{|\A_1|+|\A_2|}$,
then $\conv(\A_1) \cap \conv(\A_2)$
is a single point.
Lastly, Proskuryakov proved in
\cite{IVP} that if $\P \subset \R^d$ is a set of $d+2$
points in general position, then two points will lie in the same
component of the (unique) Radon partition of $P$ if and only if
they are separated by the hyperplane through the remaining
$d$ points. Kosmak also proved this result in \cite{LK} using
affine varieties. These properties of Radon partitions are
equivalent to Lawson's First and Second Theorems from \cite{CLL}.

\begin{lawsonsfirst} 
\cite[Theorem 1]{CLL}
Let $\P = \{ v_1,\dotsm,v_{d+2} \} \subset\R^d$ and suppose $\P$
does not lie entirely in any hyperplane. There is
a partition of $\P$ 
into three sets $\A_0$, $\A_1$, and $\A_2$, 
and $\alpha_i \in \R$, satisifying
\begin{eqnarray}
\label{uniqueintersection}
  \sum_{v_i\in \A_1} \alpha_i v_i & = & \sum_{v_i\in \A_2} \alpha_i v_i,\\
\label{sumisone}
  \sum_{v_i\in \A_1} \alpha_i & = &\sum_{v_i\in \A_2} \alpha_i = 1,\\
%\label{aiiszero}
%  \alpha_i = 0 & \text{if} & v_i\in \A_0,\\
\label{aiispositive}
  \alpha_i > 0 & \text{if} & v_i\in \A_1 \cup \A_2.
\end{eqnarray}
The numbers $\alpha_i$ are uniquely determined by the set $\P$. 
We set $\alpha_i = 0$ if $v_i \in \A_0$.
The sets $\A_0$ and $\{ \A_1, \A_2 \}$ are also unique.
\end{lawsonsfirst}

Since the sets $\A_1$ and $\A_2$ in Lawson's First Theorem
form the unique primitive Radon partition in $\P$, the point
\[ \conv(\A_1) \cap \conv(\A_2) = \sum_{v_i \in \A_1} \alpha_i v_i =
\sum_{v_i\in \A_2} \alpha_i v_i, \]
is in the relative interior of $\conv(\A_1)$ and $\conv(\A_2)$
by (\ref{uniqueintersection}), (\ref{sumisone}), 
and (\ref{aiispositive}).
These sets also determine the possible triangulations of 
$\conv(\P)$.

\begin{lawsonssecond} \cite[Theorem 2]{CLL}
Let $\P = \{ v_1,\dotsm, v_{d+2}\} \subset \R^d$,
and let $T_i$ be the simplex with vertex set
$\V(T_i) =\P\backslash \{ v_i \}$. 
There are at most two distinct triangulations of $P=\conv(\P)$,
namely 
\[ \T_1 = \{ T_i : v_i\in \A_1 \} \quad\text{and}\quad
\T_2 = \{ T_i : v_i \in \A_2 \}, \]
where the sets $\A_1$ and $\A_2$ are as defined in
Lawson's First Theorem.
The set  $\T_j$
is a valid triangulation
if and only if $|\A_j|>1$, where $j\in \{1,2\}$.
\end{lawsonssecond}

\begin{corollary} \label{lawsonscorollary}
\cite[Corollary 1]{CLL}
In the context of building triangulations, an enumeration of all possible
configurations of $d+2$ points in $\R^d$,
not lying in any hyperplane, is given by all of the possible ways
of assigning values to $|\A_0|$, $|\A_1|$, and
$|\A_2|$ satisfying
\begin{eqnarray}
\label{aoneisatleastone}
  |\A_2| & \ge & |\A_1| \ge  1,\\
\label{atwoisatleasttwo}
  |\A_2| & \ge & 2,\\
\label{azeroisatleastzero}
  |\A_0| & \ge & 0,\\
\label{sumofaisdplusone}
  |\A_0| + |\A_1| + |\A_2| & = & d+1.
\end{eqnarray}
\end{corollary}

Consider the possible sets of five points in $\R^3$
such that the convex hull is a nondegenerate
polyhedron with two different triangulations.
Lawson's Second Theorem implies
$|\A_1| \ge 2$. Inequalities
(\ref{aoneisatleastone}), (\ref{atwoisatleasttwo}),
(\ref{azeroisatleastzero}), and (\ref{sumofaisdplusone}) imply
either
\[ (|\A_0|,|\A_1|, |\A_2|) = (0,2, 3) \quad \text{or}\quad
(|\A_0|,|\A_1|, |\A_2|) = (1,2, 2). \]

\begin{example} \label{bitetraexample}
Suppose $\P = \{ v_1, \dotsm, v_5 \} \subset \R^3$
is in general position. As shown on the left in Figures \ref{bistellara}
and \ref{bistellarb}, let $P = \conv(\P)$.
In the context of Lawson's First Theorem,
$\A_0 = \emptyset$,
$\A_1 = \{ v_1, v_5 \}$, and $\A_2 = \{ v_2, v_3, v_4 \}$.
The polyhedron $P$ is a ``bipyramid'' with triangulations
\[ \T_1 = \{\ \conv(v_1, v_2, v_3, v_4),\ \conv(v_2, v_3, v_4, v_5)\ \} \]
and
\[ \T_2 = \{\ \conv(v_1, v_2, v_3, v_5),\
 \conv(v_1, v_2, v_4, v_5),\ \conv(v_1, v_3, v_4, v_5)\ \} \]
by Lawson's Second Theorem.
\end{example}

\begin{figure}[h]
\centering
\psfrag{v1}{\smaller $v_1$}
\psfrag{v2}{\smaller $v_2$}
\psfrag{v3}{\smaller $v_3$}
\psfrag{v4}{\smaller $v_4$}
\psfrag{v5}{\smaller $v_5$}
\psfrag{P}{\smaller $P$}
\psfrag{T(P)}{\smaller $\T_1$}
\includegraphics[scale=1]{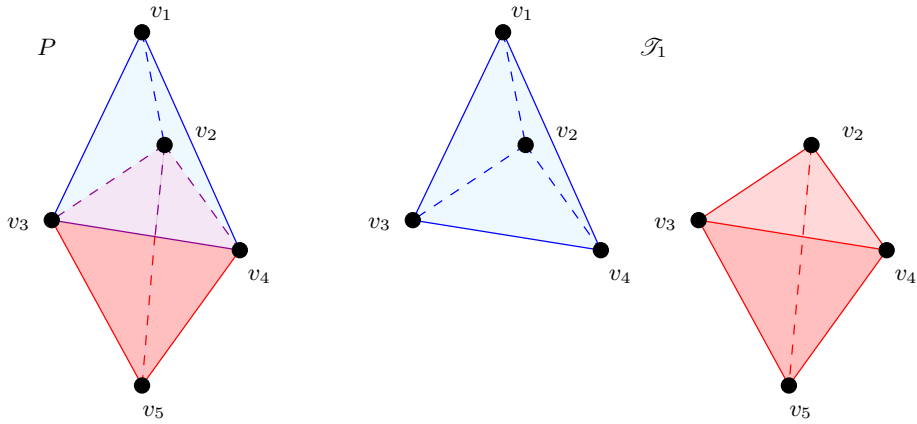}
\caption{$P$ (left) and the two simplices of $\T_1$ (right)\label{bistellara} }
\end{figure}

\begin{figure}[h] 
\centering
\psfrag{v1}{\smaller $v_1$}
\psfrag{v2}{\smaller $v_2$}
\psfrag{v3}{\smaller $v_3$}
\psfrag{v4}{\smaller $v_4$}
\psfrag{v5}{\smaller $v_5$}
\psfrag{P}{\smaller $P$}
\psfrag{T(P)}{\smaller $\T_2$}
\includegraphics[scale=1]{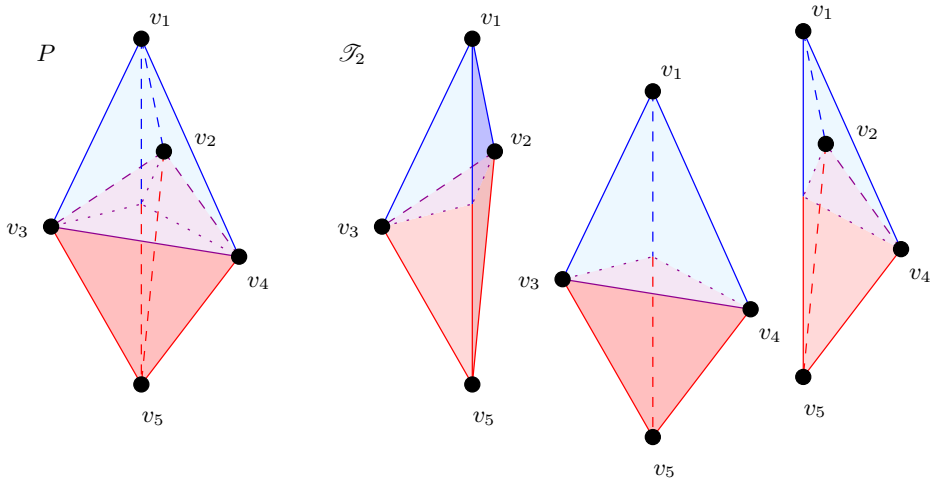}
\caption{$P$ (left) and 
and the three simplices of $\T_2$
(right) \label{bistellarb} }
\end{figure}

\begin{example} \label{pyramidexample}
Consider $\P = \{ v_1, \dotsm, v_5 \} \subset \R^3$
and $P=\conv(\P)$ shown on the left in 
Figure \ref{pyramid}.
The polyhedron $P$
is a nondegenerate  pyramid
whose base is a planar quadrilateral with vertex set
$\{ v_2, v_3, v_4, v_5\}$.
In the context of Lawson's First Theorem,
$\A_0 = \{ v_1 \}$,
$\A_1 = \{ v_2, v_4 \}$ and $\A_2 = \{ v_3, v_5 \}$.
Lawson's Second Theorem implies $P$ has two 
triangulations
\[ \T_1 = \{\ \conv(v_1, v_3, v_4, v_5),\
\conv(v_1, v_2, v_3, v_5)\ \} \]
and
\[ \T_2 = \{\ \conv(v_1, v_2, v_4, v_5),\
\conv(v_1, v_2, v_3, v_4)\ \}. \]
Each triangulation is induced
by the triangulation (in dimension 2) of the base.
\end{example}

\begin{figure}[h] 
\centering
\psfrag{v1}{\smaller $v_1$}
\psfrag{v2}{\smaller $v_2$}
\psfrag{v3}{\smaller $v_3$}
\psfrag{v4}{\smaller $v_4$}
\psfrag{v5}{\smaller $v_5$}
\psfrag{P}{\smaller $P$}
\psfrag{T1(P)}{\smaller $\T_1$}
\psfrag{T2(P)}{\smaller $\T_2$}
\includegraphics[scale=1]{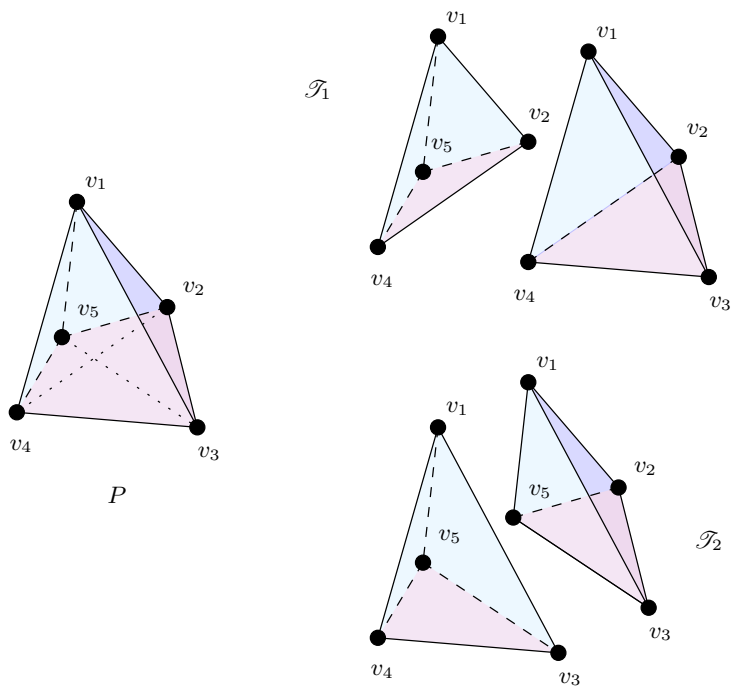}
\caption{$P$ (left) and both of its
triangulations (right) \label{pyramid} }
\end{figure}

In both Examples \ref{bitetraexample} and \ref{pyramidexample},
$|\A_1| = 2$, and  $\conv(\A_1 \cup \A_2)$
can be viewed as two simplices joined at a common facet. 
Moreover, the line segment formed by the vertices not in 
the common facet intersects that facet at a point. Such
a configuration is called a bipyramid.
In general, a $d$-polytope $P$
is a {\em $d$-bipyramid} if $P$ is the convex hull of
a line segment $L$ and a $(d-1)$-simplex $S$ such that
the intersection $L\cap S$ is a single point contained in
 $\int(L)\cap\int(S)$.
A $1$-bipyramid is simply a line segment 
with a point in its interior (or the 
convex hull of three collinear points).
In Example \ref{bitetraexample}, $P$ itself is a 3-bipyramid,
whereas  in Example \ref{pyramidexample}
$P$ contains a 2-bipyramid (its base).
The following result is a direct consequence of Lawson's theorems.

\begin{corollary} \label{hasbipyramid}
Let $P$ be a $d$-polytope with $d+2$ vertices. If there exists a
subset $\P\subseteq \V(P)$ such that $\conv(\P)$ is 
a $j$-bipyramid, where $2\le j \le d$, then $\V(P)$ 
has exactly triangulations,
one with cardinality 2 and another
with cardinality $j$.
\end{corollary}

%%%%%%%%%%%%%%%%%%%%%%%%%%%%%%%%%%%%%%%%%%%%%%%%%%%%%%%%%%%%%%%%%%%%%%%%%%%%
%
% BIPYRAMIDS
%
%%%%%%%%%%%%%%%%%%%%%%%%%%%%%%%%%%%%%%%%%%%%%%%%%%%%%%%%%%%%%%%%%%%%%%%%%%%%

\section{Existence of Bipyramids in $T\in \Skd$}

For $k\ge 2$, under what conditions will $T\in \Skd$
contain a bipyramid? Consider the following construction of possible
bipyramids from a $d$-simplex.
Let $S$ be a $d$-simplex, and let $\V(S) = \{ v_1,\dotsm, v_{d+1}\}$.
For any point $x$ exterior to $S$, let
$(\alpha_i)$ be the barycentric coordinates of $x$ relative to $S$.
By the definition of a bipyramid,
$P=\conv(S,x)$ is a bipyramid if and only if
exactly one $\alpha_i$ is negative, and the remainder are positive.
If we allow for, say, $j$ of the $\alpha_i$ to be zero, then 
$P$ will contain a  $(d-j)$-bipyramid.

The geometric interpretation is as follows. Let $H_i$ be the
hyperplane containing the facet of $S$ opposite $v_i$.
Then $P$ is a $d$-bipyramid if $x$ and $S$
are on opposite sides of exactly one $H_i$, and $x$ is on
the same side of the remaining hyperplanes as $S$.
If $x$ is  contained in $j$ of these hyperplanes,
then $P$ contains a $(d-j)$-bipyramid.

\begin{theorem} \label{musthavejbipyramid}
Suppose $T\in \Skd$, where $k\ge 2$.
Let $w_1$ and $w_2$ be any two interior lattice points of $T$,
and let $\T_{w_1} = \{ T_i \}$
be the basic triangulation of $T$ with respect to $w_1$.
If $w_2$ lies in the relative interior of a simplex in $\T_{w_1}$, say $T_{n}$, then
there exist numbers
$\alpha_i$  such that
\[ w_2 = \alpha_1 w_1 + \alpha_n v_n + 
\sum_{\substack{i=2\\i\not=n}}^{d+1}
\alpha_i v_i, 
\quad \quad \sum_{i=1}^{d+1} \alpha_i = 1, 
\]
$\alpha_1>0$, $\alpha_n < 0$, and $\alpha_i \ge 0$ otherwise.
\end{theorem}

\begin{proof}
Let $\V(T) = \{v_1,\dotsm, v_{d+1} \}$ and $F_i$ be the facet of $T$
opposite the vertex $v_i$. We may assume without loss of generality
that  $T_i = \conv(F_i, w_1)$ and
$n=d+1$ (i.e. $w_2\in \int(T_{d+1})$).
Let $L$ be the line through
$v_{d+1}$ and $w_1$. 

{\bf Case 1:} If $w_2$ lies on $L$,
then $w_1$ is necessarily between $v_{d+1}$ and $w_2$.
There exists  $\alpha > 1$ such that
\[ w_2 = v_{d+1} + \alpha (w_1 - v_{d+1})
= \alpha w_1 + (1-\alpha) v_{d+1}. \]
Choose  $\alpha_1 = \alpha$, $\alpha_{d+1} = 1-\alpha$, and
$\alpha_i = 0$ for $2\le i \le d+1$.

{\bf Case 2:} Suppose
$w_2$ does not lie on $L$. Let $x = L \cap F_{d+1}$.
If $x\not\in \int(F_{d+1})$, 
then the line segment $v_{d+1} x$,
which contains $w_1$, would be contained
in a facet of $T$ and contradict the cleanliness of $T$.
Thus $x\in \int(F_{d+1})$ and consequently
$x$ cannot be a lattice point.
Since $x\in \int(F_{d+1})$ if and only if there exists $\beta_i$
such that 
\[ x = \sum_{i=1}^d \beta_i v_i,
\quad \beta_i > 0, \quad \text{and}\quad 
\quad \text{and}\quad \sum_{i=1}^d \beta_i = 1, \]
the set
$\P = \V(T_{d+1})\cup \{ x \} \not\subset \Z^d$
can be partitioned into
\[ \A_0 = \{ w_1 \}, \quad \A_1 = \{ x \},
\quad\text{and}\quad \A_2 = \V(F_{d+1}) \]
according to Lawson's First Theorem.  By Lawson's Second Theorem,
$\P$ has exactly one (non-lattice) triangulation
\[ \T = \{\ \conv(\V(T_{d+1}) \backslash \{ v_i \} \cup \{ x \})\ :\
 1\le i \le d\ \}.\]
Since
$T_{d+1} = \conv(\P)$ and $w_2 \in \int(T_{d+1})$, $w_2$ must
be in some (non-lattice) simplex in $\T$. Without loss of generality,
suppose 
\[
w_2 \in \conv(v_2,\dotsm, v_{d}, w_1, x). \]
Then 
there exist $\gamma_i$ such that $\gamma_i \ge 0$ for $1\le i \le d+1$,
\[ \sum_{i=1}^{d+1} \gamma_i = 1, \quad \text{and} \quad
w_2 = \gamma_1 w_1 + \sum_{i=2}^d \gamma_i v_i + \gamma_{d+1} x. \]
Since $w_2$ cannot lie in any face of $T$, $\gamma_1 >0$. Similarly,
$\gamma_{d+1}>0$ since $w_2$ is not in any face of $T_{d+1}$.
Furthermore, the assumption that $w_2$ does not lie on $L$
implies one of the remaining $\gamma_i$ ($2\le i \le d)$ must also
be positive. 
Since $w_1$ lies between $v_{d+1}$ and $x$ on $L$, there exists $\mu>1$
such that
\[ x = v_{d+1} + \mu (w_1 - v_{d+1}) 
= \mu w_1 + (1-\mu) v_{d+1}. \]
It follows that
\[ w_2 = (\gamma_1 + \mu \gamma_{d+1}) w_1 + \sum_{i=2}^d \gamma_i v_i 
+ (1-\mu)\gamma_{d+1}v_{d+1}. \]
Let 
\[ \alpha_1 = \gamma_1 + \mu \gamma_{d+1},
\quad \alpha_{d+1} = (1-\mu)\gamma_{d+1}, \quad\text{and}\quad
\alpha_i = \gamma_i \quad\text{for}\quad 2\le i \le d. \]
Note that
\[ \sum_{i=1}^{d+1} \alpha_i
= [\mu + (1-\mu)]\gamma_{d+1} + \sum_{i=1}^{d} \gamma_i = 1, \]
$\alpha_1 > 0$, and $\alpha_{d+1} < 0$. The remaining
$\alpha_i$ are nonnegative, and at
least one is positive since at least one
of the $\gamma_i$ is  positive for $2\le i \le d$.
\end{proof}

The geometric interpretation of Theorem \ref{musthavejbipyramid}
is that for some simplex $T_{m} \not= T_n$ in 
$\T_{w_1}$, $\conv(T_m, w_2)$ 
contains a $j$-bipyramid $P$,
and $\V(P)$ includes $w_1$ and $w_2$. 
In terms of triangulations, $\V(T) \cup \{ w_1 , w_2 \}$
has two  triangulations, provided $j>1$.

\begin{corollary} \label{thastwotriangulations}
Suppose $T\in \Skd$, where $k\ge 2$.
If $w_2$ lies in the relative interior of a simplex in 
the basic triangulation $\T_{w_1}$ of $T$,
then the set $\V(T)\cup \{ w_1 , w_2 \}$ has two 
triangulations provided $w_1$ and $w_2$
are not collinear with any $v\in \V(T)$.
\end{corollary}

\begin{proof}
Let $\T_1 = \T_{w_1} = \{ T_i \}$.
Without loss of generality, suppose $w_2 \in \int(T_{d+1})$. 
Let $\T_{w_2}$ be the basic triangulation of $T_{d+1}$
with respect to $w_2$.
One  triangulation of $\V(T) \cup \{ w_1, w_2 \}$
is the refinement $\T_2$ of $\T_1$ guaranteed by 
Theorem \ref{refinementtheorem} (with $w=w_2$), where
\[ \T_2 = \T_{w_1} \backslash \{ T_{d+1} \} \cup \T_{w_2}. \]
Let $S_i$ be the facet of $T_{d+1}$ opposite $v_i$ for $1\le i \le d$,
and let $S_{d+1}$ be the facet of $T_{d+1}$ opposite $w_1$.
Theorem \ref{musthavejbipyramid} guarantees there exists
$1\le m \le d$ such that
\[\conv(T_m, w_2) =  \conv\left(
T_m \bigcup 
\conv(S_m, w_2) \right). \]
contains a $j$-bipyramid.
Note that both $T_m$ and $\conv(S_m, w_2)$ are simplices in 
$\T_{2}$.
 Since 
$w_1$ and $w_2$ are not collinear with any $v\in \V(T)$,
$T$ is clean, and no three vertices of $T$ are collinear,
it follows that
$j>1$. Corollary \ref{hasbipyramid} implies $\V(T_m)\cup \{ w_2 \}$
has two triangulations, one of which is contained
in $\T_2$. That is,
\[ \conv(T_m, w_2) = T_m\ \bigcup\ \conv(S_m, w_2). \]
Thus $\V(T) \cup \{ w_1, w_2 \}$
has two triangulations.
\end{proof}

Among all possible $d$-bipyramids that have
two triangulations, 2-bipyramids are unique in that their
triangulations have the same
cardinality. The existence of lattice 2-bipyramids within lattice
$d$-simplices has further implications.
Note that 2-bipyramids are simply convex planar 
quadrilaterals.

\begin{lemma} \label{noparallelogram}
Suppose $T\in S_k^d$ where $k\ge 2$. Let $w_1$ and $w_2$ be any two of
the interior points of $T$. Let $Q = \conv( w_1, w_2, v, v' )$ where
$v, v' \in \V(T)$.
If $Q$ is a planar quadrilateral, then 
the opposing edges of $Q$ cannot be parallel.
\end{lemma}

\begin{proof}
Without loss of generality, suppose $w_1 v$ and $w_2 v'$ are edges of $Q$.
Let $H_1$ be the hyperplane containing 
the facet of $T$ opposite $v'$,
and let $H_2$ be the hyperplane parallel to $H_1$ and
containing $v'$. Since $w_1$ is an interior point of $T$, it must lie
between $H_1$ and $H_2$. 
\begin{figure}[h]
\centering
\psfrag{H1}{\smaller $H_1$}
\psfrag{H2}{\smaller $H_2$}
\psfrag{w1}{\smaller $w_1$}
\psfrag{w2}{\smaller $w_2$}
\psfrag{v}{\smaller $v$}
\psfrag{v'}{\smaller $v'$}
\psfrag{w}{\smaller $w$}
\includegraphics[scale=1.0]{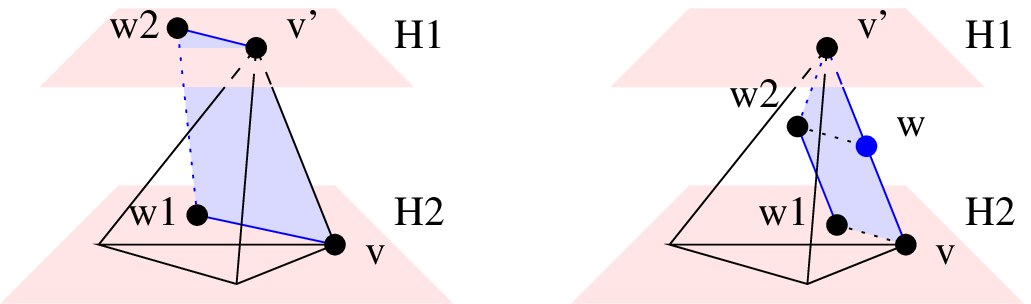}
\end{figure}
If $w_1 v \parallel w_2 v'$, then $w_2$ 
must be opposite of $w_1$ relative to both $H_1$ and $H_2$.
This is impossible since $w_2$ would lie outside of $T$ as
shown in the left figure above. If $w_1 w_2 \parallel v v'$, 
edge $w_1 w_2$ of $Q$ must be shorter than edge $v v'$ since
$w_1 w_2$ is in the interior of $T$.
Then  $w=v' - (w_2 - w_1)$, or $w=v + (w_2 - w_1)$,
is a lattice point on the edge $v v'$, as shown
above on the right, which is also impossible as $T$
is clean. 
\end{proof}

\begin{lemma} \label{qhaslatticepoint}
Suppose $Q = \conv( v_1,v_2, v_3, v_4 )$ is a planar lattice
quadrilateral.
If the opposing edges of $Q$ are not
parallel, then interior of $Q$ contains a lattice point
$w \not= v_i$ for $1\le i \le 4$. Moreover, $w$ lies in the interior of
a triangle whose edge set is a subset of the edges
and diagonals of $Q$.
\end{lemma}

\begin{proof}
Consider the pairs of adjacent edges of $Q$. 
Without loss of generality, suppose $v_1 v_2$ and
$v_2 v_3$ are edges of $Q$ such that 
the triangle $\Delta = \conv( v_1, v_2, v_3)$ has
minimal area. The lines containing each edge of $\Delta$ partition the plane
into 7 different regions
as shown below (left). 
\begin{figure}[h]
\begin{center}
\psfrag{v1}{\smaller $v_1$}
\psfrag{v2}{\smaller $v_2$}
\psfrag{v3}{\smaller $v_3$}
\psfrag{I}{\smaller $I$}
\psfrag{II}{\smaller $II$}
\psfrag{III}{\smaller $III$}
\psfrag{IV}{\smaller $IV$}
\psfrag{H1}{\smaller $H_1$}
\psfrag{H2}{\smaller $H_2$}
\psfrag{L1}{\smaller $L_1$}
\psfrag{L2}{\smaller $L_2$}
\psfrag{R}{\smaller $R$}
\includegraphics[scale=1.0]{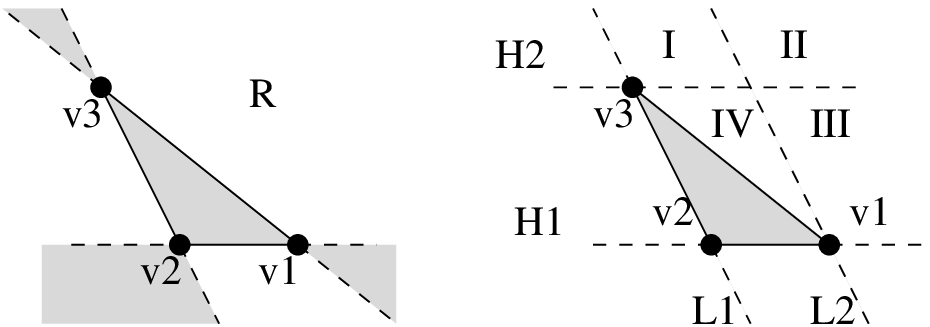}
\end{center}
\end{figure}
Since $Q$
is a convex quadrilateral, $v_4$ can only 
lie in the three unshaded regions.  Without loss of generality,
suppose $v_4$ lies in region $R$ as shown above.
Let $H_1$ be the line through $v_1 v_2$ and
$H_2$ be the line through $v_3$ such that $H_2 \parallel H_1$.
Similarly, let $L_1$ be the line through $v_2 v_3$ and 
$L_2$ be the line through $v_1$ such that $L2 \parallel L1$.
(See right-side figure above.) These four lines divide $R$
into four regions. Since $Q$ has no parallel edges, $v_4$ cannot
lie on any one of these four lines.
 If $v_4$ lies in region I,
then $\conv( v_2, v_3, v_4)$ is a triangle of smaller
area than $\Delta$ (they both share the same leg $v_2 v_3$,
whereas the distance between $v_4$ and $L_1$ is shorter than
the distance between $v_1$ and $L_1$).
Similarly, $v_4$ cannot lie in regions III or IV. 
Thus $v_4$ must lie in region II. Then
$w = v_1 + (v_3 - v_2)$
is a lattice point contained in interior of $\conv(v_1, v_3, v_4)$.
\end{proof}

%%%%%%%%%%%%%%%%%%%%%%%%%%%%%%%%%%%%%%%%%%%%%%%%%%%%%%%%%%%%%%%%%%%%%%%%%%%%
%
% COLLINEARITY
%
%%%%%%%%%%%%%%%%%%%%%%%%%%%%%%%%%%%%%%%%%%%%%%%%%%%%%%%%%%%%%%%%%%%%%%%%%%%%

\section{Collinearity Property of Minimal Volume
$T\in \Skd$}

We can now prove the collinearity property of $\int(T)\cap \Z^d$,
where $T\in \Skd$, $d\ge 3$, and $\Vol(T) = \frac{1}{d!}(dk+1)$.

\begin{theorem} \label{twocollinear}
Suppose $T\in \Skd$, where $d\ge 3$ and
$\Vol(T) = \frac{1}{d!}(dk+1)$.
Any two interior lattice points of $T$
are collinear with a vertex of $T$.
\end{theorem}

\begin{proof}
The claim is vacuously true for $k\le 1$, so suppose $k\ge 2$.
If $k=2$ then let $w_1$ and $w_2$ be the two interior lattice points of $T$.
Otherwise let $w_1$ and $w_2$ be any two interior lattice points
of $T$.
Suppose $w_1$ and $w_2$ are not collinear with any
$v\in \V(T)$. That is, suppose there does not
exist a lattice 1-bipyramid consisting of $w_1$, $w_2$, and
any $v\in \V(T)$.
Let $\T_1 = \T_{w_1} = \{ T_i \}$, where 
$\T_{w_1}$ is the basic triangulation of $T$ with respect to $w_1$.
By assumption, $w_2$ must lie
in the relative interior of either a simplex in $\T_1$,
or a $j$-face  of a simplex in $\T_1$, where $1<j<d$.

{\bf Case 1:} If $w_2$ lies in the relative
interior of a $j$-face of a simplex
in $\T_1$, where $1<j<d$, then Corollary \ref{nojsimplex}
implies $\Vol(T) > \frac{1}{d!}(dk+1)$, a contradiction.

{\bf Case 2:} If $w_2$ lies in the relative interior of 
a simplex in $\T_1$, then we may assume without loss of
generality that $w\in \int(T_{d+1})$. Let $\T_{w_2}$
be the basic triangulation of $T_{d+1}$ with respect to $w_2$.  
The refinement $\T_2$ of $\T_1$ obtained by applying Theorem
\ref{refinementtheorem} with $w=w_2$ is
\[ \T_2 = \T_{w_1} \backslash \{ T_{d+1} \} \cup \T_{w_2}. \]
Theorem \ref{musthavejbipyramid}
implies $T$ contains a lattice $j$-bipyramid $P$,
where $2 \le j \le d$ and $\{ w_1, w_2 \} \subset \V(P)$.
If $j=2$, then 
$P = \conv(v, v', w_1, w_2)$ is convex planar lattice
quadrilateral,
where $\{ v, v' \} \subset \V(T)$. Lemmas 
\ref{noparallelogram} and \ref{qhaslatticepoint} imply
$P$ contains a lattice point 
$w_3 \not\in \{ w_1, w_2 \}$.
This is a contradiction if $k=2$. If $k>2$,
then the second part of Lemma \ref{qhaslatticepoint}
implies $w_3$ lies in the relative interior of
some triangle $\Delta_i$, where
\begin{align*}
\Delta_1 & = \conv(w_1, w_2, v), &
\Delta_2 & = \conv(w_1, w_2, v'),\\
\Delta_3 & = \conv(w_1, v, v'), &
\Delta_4 & = \conv(w_2, v, v').
\end{align*}
Since $d\ge 3$, and  each $\Delta_i$ 
a 2-face of simplices in $\T_2$,
$\Vol(T) > \frac{1}{d!}(dk+1)$ by
Corollary \ref{nojsimplex}. This contradicts the mimal volume
property of $T$.
On the other hand, if $j>2$, then $\V(T) \cup \{ w_1, w_2 \}$
has two possible triangulations $\T_2$ and $\T_2'$
where  
\[ |\T_2'| = |\T_1| + d-2+j > 2d+1. \]
Subsequent applications of Theorem \ref{refinementtheorem}, starting with
$\T_2'$, and (\ref{volbytriangulation})  imply
\[ \Vol(T)>\frac{1}{d!}(dk+1). \]
This is again a contradiction. Thus any two points in $\int(T)\cap \Z^d$ must
be collinear with some $v\in \V(T)$.
\end{proof}

\begin{corollary} \label{allcollinear}
Suppose $T\in \Skd$, where $d\ge 3$
and $\Vol(T) = \frac{1}{d!}(dk+1)$. The
points $\int(T)\cap \Z^d$ are collinear with some
$v\in \V(T)$.
\end{corollary}

\begin{proof}
The claim is vacuously true for $k \le 1$. Theorem
\ref{twocollinear} implies the claim also holds for $k=2$.
Suppose $k\ge 3$.
It suffices to show the claim holds for any three points in $\int(T)\cap \Z^d$.
Let $w_1$, $w_2$, and $w_3$ be any three points in $\int(T)\cap \Z^d$.
Theorem \ref{twocollinear} implies there exists 
$v\in \V(T)$ such that $w_1$, $w_2$, and $v$ are collinear
Similarly, there exists  $v'\in \V(T)$ such that
$w_1$, $w_3$, and $v'$ are collinear. We need to show $v=v'$.
Suppose $v\not=v'$. Let $L$ be the line through $w_1$, $w_2$ and $v$,
and let $L'$ be the line through $w_1$, $w_3$, and $v'$.
Without loss of generality, suppose $w_2$ is between $w_1$ and
$v$ on $L$.

{\bf Case 1:} If $w_3$ lies between $w_1$ and $v'$ on $L'$, then
$P = \conv(v, v', w_2, w_3)$ is a convex planar quadrilateral
contained within $\conv(w_1, v, v')$. Lemmas
\ref{noparallelogram} and \ref{qhaslatticepoint}
imply $\int(P)$ contains a lattice point $w_4 \not\in\{ w_1, w_2, w_3 \}$.
Since $P \subset \conv(w_1, v, v')$ and $d\ge 3$, Corollary
\ref{nojsimplex} implies $\Vol(T)>\frac{1}{d!}(dk+1)$.

{\bf Case 2:} If $w_1$ lies between $w_3$ and $v'$ on $L'$, then
$w_2 \in \int(\conv(w_3, v, v'))$.
 Again, since $d\ge 3$, Corollary 
\ref{nojsimplex} implies $\Vol(T)> \frac{1}{d!}(dk+1)$.
Thus $v = v'$, and all points in $\int(T)$ and $v$ are collinear.
\end{proof}

\begin{corollary} \label{evenlyspaced}
Suppose $T\in \Skd$ and $d\ge 3$.
Let $L$ be the line through $\int(T)\cap \Z^d$.
The consecutive points on $L$ are evenly spaced.
\end{corollary}

\begin{proof}
Let $v = L \cap \V(T)$ and let $w \in \int(T)$ be the lattice point
closest to $v$.
Since any line can be described as a linear combination of two points,
any $w'\in L$ can be expressed as
\begin{equation} \label{equationofline}
w' = v + \alpha (w-v),
\end{equation}
where $\alpha \in \R$. 
It suffices to show
that for any $w'\in L\cap T \cap \Z^d$, the
$\alpha$ in (\ref{equationofline}) is an integer and
$0 \le \alpha \le k$.
If $\alpha < 0$, then $v\in \int(\conv(w,w')) \subset \int(T)$,
which is impossible.
If $\alpha = 0$ then $w'=v$.
Let $[a]$ and $\{ \alpha \}$ denote the 
integer and fractional parts 
of $\alpha$, respectively. If $\{\alpha \} \not=0$, then
\[ x = w' - [\alpha] (w-v) = v+\{ \alpha \}(w-v) \in \Z^d \]
lies between $v$ and $w$ on $L$, which
contradicts our assumption that $w$ is closest to $v$ on $L$.
Hence $\alpha \in \Z$.
Finally, if $\alpha > k$, then $T$ has more than $k$ interior
points, which is aslo impossible. 
\end{proof}

\begin{corollary} \label{uniquetriangulation}
If $T\in \Skd$, $\Vol(T) = \frac{1}{d!}(dk+1)$, 
and $d \ge 3$, then
$T\cap \Z^d$ has a unique triangulation.
\end{corollary}

\begin{proof}
Let $\V(T) = \{ v_1,\dotsm,v_{d+1} \}$.  Without loss of
generality, suppose $v_{d+1}$ is collinear with
$\int(T)\cap \Z^d$.
Let $w_1$ be the point in  $\int(T)\cap \Z^d$ closest to $v_{d+1}$
and 
\[ \mathscr{W} = \{\ w_i\ :\ w_i 
= v_{d+1} + i (w_1-v_{d+1}),\ 0 \le i \le k\  \}. \]
Let $S_{k+1,d} = \conv(v_1,\dotsm, v_d, w_k)$. 
For $1\le i \le k$ and $1\le j \le d$, let
\[ S_{i,j} = \conv(\V(T)\backslash \{ v_j, v_{d+1} \} \cup \{
w_i, w_{i-1}\}). \]
It is easy to check that $\T = \{ S_{i, j} \}$ is a triangulation of $T\cap \Z^d$.
In fact, this triangulation corresponds to the refinement sequence
in Corollary \ref{simplextriangulation}.

Let $\T'$ be another full triangulation of $T$.
Consider any $S\in \T'$. Since $\T'$ is a full
triangulation, $S$ is necessarily empty and clean.
These two conditions force
\[1 \le |\V(S) \cap \mathscr{W}| \le 2. \]

{\bf Case 1:} If $|\V(S)\cap \mathscr{W}| = 1$, then there
exist $v\in \V(T)$ and $w \in \mathscr{W}$ such that
\[ S = \conv(\V(T)\backslash \{ v \} \cup \{ w \}). \]
Since $S$ is empty, it follows that $w=w_k$, $v=v_{d+1}$,
and $S = S_{k+1,d}$.

{\bf Case 2:} If $|\V(S) \cap \mathscr{W}| = 2$, then there
exist $\{ v, v' \} \subset \V(T)$ and $\{ w, w' \} \subset \mathscr{W}$
such that 
\[ S = \conv(\V(T) \backslash \{ v, v' \} \cup \{ w , w' \}). \]
Without loss of generality, suppose $w$ is closer to $w_0$ than $w'$ is to 
$w_0$. Since $S$ is clean, 
\begin{equation}\label{ithlevel}
(w,w') = (w_{i-1}, w_{i})
\end{equation} for some
$1 \le i \le k$. Moreover, $w_0$ ($= v_{d+1}$) is a vertex of $S$ if
and only if $i=1$ in (\ref{ithlevel})
since $S$ must also be empty. 
Thus $S = S_{i,j}$ for some appropriate $j$,
and $\T'\subseteq \T$. 
Consequently, $\T'=\T$.
\end{proof}

%%%%%%%%%%%%%%%%%%%%%%%%%%%%%%%%%%%%%%%%%%%%%%%%%%%%%%%%%%%%%%%%%%%%%%%%%%%%
%
% UNIMODULAR TRANSFORMATIONS
%
%%%%%%%%%%%%%%%%%%%%%%%%%%%%%%%%%%%%%%%%%%%%%%%%%%%%%%%%%%%%%%%%%%%%%%%%%%%%

\section{Unimodular Transformations}

A {\em unimodular transformation} $f:\Z^d \to \Z^d$ is an affine
map of the form
\[ f(v) = v \cdot \matrixm  + u, \]
where $u,v \in \Z^d$, $\matrixm$ is a $d\times d$ matrix with
entries in $\Z$, and $\det(\matrixm) = \pm 1$. A {\em translation}
is a unimodular transformation in which $\matrixm = \matrixi_d$,
where $\matrixi_d$ is the $d\times d$ identity matrix.
If $\det(\matrixm) = \pm 1$, $\matrixm$ is invertible, 
then $\matrixm^{-1}$ is unimodular, and
$f$ is one-to-one. Thus $P$ is a lattice $d$-polytope if and only if
$f(P)$ is a lattice $d$-polytope with the same number of vertices.
Using barycentric coordinates, it is easy to check that
$f(w)$ is
an interior 
point of $f(P)$ if and only if
$w$ is
an interior point of $P$. Similarly, $f(w)$
is a boundary point of $f(P)$ if and only if $w$ is a boundary
point of $P$ (cf. \cite{BR}).
We say
$P_1$ and $P_2$ are {\em equivalent} lattice $d$-polytopes 
(and write $P_1\simeq P_2$) if there exists a unimodular transformation
 $f$ such that
$P_1 = f(P_2)$.

\begin{lemma} \label{unitsimplex}
Let $T$ be a $d$-simplex with $\Vol(T) = \frac{1}{d!}$. Then
$T$ is equivalent to $T_{1,\dotsm,1}$, the $d$-simplex
whose vertex set consists of the origin, $e_i$ for $1\le i \le d-1$,
and the point $(1,\dotsm,1)$.
\end{lemma}

\begin{proof}
Let $v_1, \dotsm, v_{d+1}$ be an enumeration of the vertices of 
$T$. If necessary, translate $T$ so that one of its vertices is
the origin. Without loss of generality, suppose $v_{d+1}=\origin$.
Let
\[ P = \conv\left(v_1,\dotsm,v_{d+1}, \sum_{i=1}^{d} v_i\right). \]
Note that $P$ is a parallelepiped containing $T$,
and $\Vol(P) = \pm \det(\matrixa)$, 
where $\matrixa$ is a $d\times d$ matrix whose entry $a_{i,j}$
is the $j$-th coordinate of $v_i$.
Since 
\[\frac{1}{d!} = \Vol(T) = \frac{1}{d!} \cdot \Vol(P) = \frac{1}{d!}
\cdot |\det(\matrixa)|, \]
$\det(\matrixa) = \pm 1$ and $\matrixa$ is invertible. Thus, there exists
a unique, unimodular, 
$d\times d$ matrix $\matrixm$ such that $\matrixa \matrixm = \matrixi_d$.
Thus $f:\Z^d\to\Z^d$ defined by $f(v) = v\cdot \matrixm$ is
a unimodular transformation such that
and $f(v_i) = e_i$ for $1\le i \le d$
and $f(v_{d+1}) = \origin$. Let $g:\Z^d \to \Z^d$ be 
the unimodular transformation given by
$g(v) = v \cdot \matrixm'$ where
\[
 \matrixm' = \left[ \begin{array}{cc}
 \matrixi_{d-1}  & 0\\
 1 & 1 
\end{array}
\right]. 
\]
It is easy to check that $g(e_i) = e_i$ for $1\le i \le d-1$,
$g(e_{d}) = (1,\dotsm,1)$, and $g(\origin) = \origin$.
Then the composition $g\circ f$, where
$(g \circ f)(v) = v \cdot \matrixm \cdot \matrixm'$,
is also a unimodular transformation, and
$(g\circ f)(T) = T_{1,\dotsm,1}$.
\end{proof}

Let $T_{a_1,\dotsm,a_d}$ denote the $d$-simplex whose vertex set
consists of the origin, the unit points $e_1,\dotsm, e_{d-1}$,
and the point $(a_1,\dotsm, a_{d})\in \Z^d$. What can we say about
$T_{a_1,\dotsm,a_d}$?
If $a_d<0$, then
we can apply  $f:\Z^d \to \Z^d$, where
\begin{equation}\label{positivefinalcoordinate}
 f(v) = v \cdot \left[
\begin{array}{cr}
\matrixi_{d-1} & 0\\
0 & -1 
\end{array}
\right].
\end{equation}
Note that $f$ is simpy a reflection in the $d$-th coordinate.
By reflection, if necessary, we can take $a_d$ to be nonnegative.
Moreover, $a_d$ is determined by the volume of $T_{a_1,\dotsm,a_d}$
since
\begin{equation} \label{asubd}
0< d!\cdot \Vol(T_{a_1,\dotsm, a_d})= \pm 
\det
\left[
\begin{array}{ccc}
a_1 & \dotsm    & a_d\\
\multicolumn{2}{c}{\matrixi_{d-1}}    & 0 
\end{array}
\right]
= 
| a_{d}|.
\end{equation}

For the remaining $a_i$, consider 
the integers $b_i$
such that such that 
\[ 0 \le a_i + b_i a_d < a_d,\] and the
transformation
$g:\Z^d\to \Z^d$, where
\begin{equation} \label{asubinonnegative}
g(v) = v \cdot \left[ 
\begin{array}{ccc|c}
& & & 0\\
 & \matrixi_{d-1} & & \vdots \\
& & & 0\\
\hline
b_1 & \dotsm & b_{d-1} & 1
\end{array}
\right].
\end{equation}
All the vertices of $T_{a_1,\dotsm,a_d}$ excluding
$(a_1,\dotsm, a_d)$ are fixed under $g$. On the other hand,
$g$ sends $(a_1,\dotsm,a_d)$ to
$(a_1',\dotsm,a_d')$
where $a_d' = a_d$ and
$a_i'=  a_i + b_i a_d$ for $1\le i \le d-1$.
Since $T_{a_1,\dotsm,a_d}$ is not contained in any
hyperplane, $a_i > 0$ and $a_i'>0$.
Thus the class of all $T_{a_1,\dotsm,a_d} \in \Skd$ is represented by
$T_{a_1',\dotsm,a_d'}$ where
\begin{equation} \label{aiarepositive}
 0 < a_i ' < a_d'=a_d.
\end{equation}

Lastly, let $w\in \int(T_{a_1,\dotsm,a_d})$ and let $(\lambda_i)$ be the
barycentric coordinates of $w$ relative to $T_{a_1,\dotsm, a_d}$, where
\[ w = \lambda_{d+1}\cdot \origin + \lambda_d \cdot (a_1,\dotsm, a_d)
+ \sum_{i=1}^{d-1} \lambda_i e_i. \]
The $d$-th coordinate of $w$ is
$\lambda_d \cdot a_d$, and for $1\le i \le d-1$, the $i$-th 
coordinate is
\begin{equation} \label{ithcoordinate}
\lambda_i  + \lambda_{d} \cdot a_i
= \lceil \lambda_{d} \cdot a_i \rceil .
\end{equation}
We can conclude that every point in $\int(T_{a_1,\dotsm,a_d})\cap \Z^d$
is completely determined by its $d$-th coordinate. Equivalently,
no two interior lattice points can have the same $d$-th
coordinate.

\begin{theorem} \label{hasgeneralform}
Let $d\ge 3$ and $k\ge 1$.
If $T\in \Skd$, and $\Vol(T) = \frac{1}{d!}(dk+1)$,
then 
$T\simeq T_{a_1,\dotsm, a_d},$
where 
$(a_1,\dotsm,a_{d-1},  a_{d}) = (dk,\dotsm, dk, dk+1)$.
\end{theorem}

\begin{proof}
Let $\V(T)$, $v_{d+1}$,
$\mathscr{W}$,  $S_{i,j}$, and $\T$ be as in the proof of Corollary
\ref{uniquetriangulation}. For $S_{1,d} \in \T$, Lemma
\ref{unitsimplex} implies there exists a unimodular transformation
$f$ such that 
$f(v_{d+1}) = \origin$, $f(w_1) = (1,\dotsm,1)$,
and $f(v_j) = e_j$ for $1\le j \le d-1$. Let
\[ f(v_d) = (a_1,\dotsm, a_{d}). \] 
We may assume $0 < a_j \le a_d$ by (\ref{aiarepositive}).
Since $\Vol(T) = \frac{1}{d!}(dk+1)$ and unimodular
transformations preserve volume, (\ref{asubd}) implies
 $a_d = dk+1$. 
Unimodular transformations also 
preserve interior lattice points, so
 $f(w_1) \in \int(f(T)) \cap \Z^d$, and
$f(w_1)$ is collinear with $\origin$.
Corollary \ref{evenlyspaced} implies
$f(w_i)  = (i,\dotsm, i)$.
If $(\lambda_{i,j})$ are the barycentric coordinates of $f(w_i)$,
where
\[ f(w_i) = \lambda_{i,d+1} \cdot \origin
+ \lambda_{i,d} \cdot (a_1,\dotsm, a_d) +
\sum_{j=1}^{d-1} \lambda_{i,j}\cdot  e_j, \]
then (\ref{ithcoordinate}) implies
\[ \lambda_{i,j} + \lambda_{i,d}\cdot a_j = i. \]
Note that $0 < \lambda_{i,j} < 1$, and
\[ f(w_i) \in \int(f(T))\cap \Z^d
\iff w_i \in \int(T)\cap \Z^d \iff 1\le i \le k.\]
Since
\[
\lambda_{i,d} \cdot a_j  = (i-1) +  1- \lambda_{i,j} 
   > (i-1)\lambda_{i,d} +  1 - 
    \sum_{\substack{j=1\\j\not=d}}^{d+1} \lambda_{i,j}
   > i\cdot \lambda_{i,d},
\]
$a_j > i$ for all $i$, which implies $a_j > k \ge 1$.
The final step is to show that $a_j = dk$.

Consider the interior point $f(w_k) = (k,k,\dotsm,k)$ and
let $(\lambda_j)$ be the barycentric coordinates of $f(w_k)$.
Clearly $\lambda_d = \frac{k}{dk+1}$.
By  (\ref{ithcoordinate}), we have
\[ \lambda_j = k - a_j \cdot \lambda_d = k - \frac{a_j \cdot k}{dk+1}. \]
Since $\ds \sum_{j=1}^{d+1} \lambda_j = 1$,
it follows that 
\[ 0 < \sum_{j=1}^d \lambda_j
 = (d-1)k - \frac{k}{dk+1}(a_1 + \dotsm + a_{d-1}-1)  < 1. \]
Using elementary algebraic manipulation, we obtain
\[ 1 \le (kd+1)(d-1) - (a_1 + \dotsm + a_{d-1}-1) \le d, \]
which implies
\begin{equation} \label{remainingaj}
a_1 + \dotsm + a_{d-1} \ge kd(d-1).
\end{equation}
Finally, $a_j < a_d = dk+1$ and (\ref{remainingaj})
imply $a_j = dk$ for $1\le j \le d$.
\end{proof}

\begin{corollary}
If $T\in \Skd$, then $\Vol(T) = \frac{1}{d!}(dk+1)$ if and only if
$T\simeq S_d(k)$.
\end{corollary}

\begin{proof}
We first show that $\Vol(S_d(k)) = \frac{1}{d!}(dk+1)$.
Using determinants, 
\[ \Vol(S_d(k)) = \frac{1}{d!}\cdot |\det(\matrixa)|, \]
where
\[ \matrixa =
\left[ 
\begin{array}{c}
e_2 - e_1 \\
e_3 - e_1 \\
\vdots \\
e_{d-1} - e_1 \\
e_d - e_1 \\
(-k,-k,\dotsm,-k)-e_1
\end{array}
\right]
= \underbrace{\left[
\begin{array}{c|ccccc}
-1                \\
-1 &  & \matrixi_{d-1} \\
-1  \\
\hline
-k-1 & -k  & \dotsm & -k
\end{array}
\right]}_{d\times d}.
\]
Let $\matrixa'$ be the matrix obtained by
adding $k$ times the sum of the first $d-1$ rows of $\matrixa$
to row $d$ of $\matrixa$. That is,
\[ \matrixa' = 
\left[\begin{array}{c|ccccc}
-1                \\
-1 &  & \matrixi_{d-1} \\
-1  \\
\hline
-dk-1 & 0  & \dotsm & 0
\end{array}
\right].
\] 
It is easy to check that $|\det(\matrixa)| = |\det(\matrixa')| = dk+1$.
If $T\simeq S_d(k)$, then $\Vol(T) = \Vol(S_d(k)) = \frac{1}{d!}(dk+1)$.
On the other hand, if $\Vol(T) = \frac{1}{d!}(dk+1)$, then
$T\simeq S_d(k)$ by Theorem \ref{hasgeneralform} provided $S_d(k)$ is clean.
Since $\Vol(S_d(k)) = \frac{1}{d!}(dk+1)$ and any lattice triangulation of
$\V(S_d(k)) \cup \int(S_d(k))\cap \Z^d$ contains at least $dk+1$ simplices,
(\ref{sizeofbasictriangulations}) and Corollary \ref{simplextriangulation}
imply $S_d(k)$ cannot have any lattice points on its boundary except for
its vertices.
\end{proof}

Let $f:\Z^d \to \Z^d$ be the translation
\[ f(v) = v \cdot \matrixi_d  + k\cdot\sum_{i=1}^d e_i = v + (k,k,\dotsm,k), \]
and let
$g:\Z^d \to \Z^d$ be the transformation $g(v) = v\cdot \matrixm$
where
\[ \matrixm =  
\underbrace{
\left[
\begin{array}{cccccc}
1-k &  -k & -k & \dotsm & -k &   -k \\
-k & 1-k & -k & \dotsm &  -k & -k \\
-k & -k & 1-k & \dotsm &  -k & -k \\
\vdots & \vdots & \vdots &\ddots & \vdots & \vdots\\
-k & -k & -k & \dotsm & 1-k & -k \\
(d-1)k & (d-1)k & (d-1)k & \dotsm & (d-1)k & (d-1)k+1 
\end{array}
\right]}_{d\times d}.
\]
That is, if $m_{i,j}$ is the entry  in the row $i$
and column $j$ of $\matrixm$, then
\[ m_{i,j} =
\begin{cases}
 1-k, & i=j\not=d,\\
(d-1)k, & j=d,\ i\not=d\\
(d-1)k+1, & i=j=d\\
-k, & \text{otherwise}.
\end{cases}
\]
Consider the image of $\V(S_d(k))$ under the composition $g\circ f$.
For $v\in \V(S_d(k))$,
the $j$-th coordinate of $v$ under $g\circ f$ is simply the
dot product of $v+(k,\dotsm,k)$ with the $j$-th column of $\matrixm$.
It is easy to check that
\[ (g\circ f)(v)
= \begin{cases}
\origin, & v = -k\sum_{i=1}^d e_i,\\e_i, & v = e_i,\ 1\le i \le d-1,\\
(a_1,\dotsm,a_d), & v=e_d.
\end{cases}
\]
The composition $g\circ f$ is in fact a unimodular transformation
which maps $S_d(k)$ to $T_{a_1,\dotsm,a_d}$, where the $a_i$
are as in Theorem \ref{hasgeneralform}.

%%%%%%%%%%%%%%%%%%%%%%%%%%%%%%%%%%%%%%%%%%%%%%%%%%%%%%%%%%%%%%%%%%%%%%%%%%%%
%
% COUNTEREXAMPLES IN R^2
%
%%%%%%%%%%%%%%%%%%%%%%%%%%%%%%%%%%%%%%%%%%%%%%%%%%%%%%%%%%%%%%%%%%%%%%%%%%%%
\section{Counterexamples in $\R^2$}

The collinearity property of the interior lattice points does not
hold for all clean triangles in $\R^2$.
Pick's theorem states that if $P$ is a lattice polygon, then
\begin{equation} \label{picks}
 \Vol(P)  = k + \frac{b}{2} - 1
\end{equation}
where $k = |\int(P)\cap \Z^2|$ and $b =|\partial P \cap \Z^2|$.
For clean triangles, 
(\ref{picks}) reduces to
\[ \Vol(P) = k+\frac{1}{2}. \]
Thus all clean lattice triangles satisfy the minimal volume condition.
In \cite{BR86}, Reznick proved that any $T\in S_k^2$ is 
equivalent to some $T_{a,2k+1}$, where $0<a<2k+1$. 
Using this representation for clean lattice triangles,
Reznick then showed that the number of equivalence classes 
of clean $k$-point lattice triangles increases with $k$.
Thus there exist clean lattice triangles
whose interior points are not collinear.

Consider the triangle $\Delta_{p,q}$ in $\R^2$ with vertex set
$\{\ (-1,0),\ (0,q),\ (p,-1)\ \}$,
where $1 \le p \le q$ and $\gcd(p,q+1) = 1$. It is easy to check
that $\Delta_{p,q}$ is clean. 
We first compute the number of interior
points of $\Delta_{p,q}$.
Using determinants, 
\[ \Vol(\Delta{p,q}) = \frac{q}{2}(p+1) + \frac{1}{2}. \]
Pick's theorem implies
\begin{equation} \label{pointsindeltapq}
|\int(\Delta{p,q}) \cap \Z^2| = \frac{q}{2}(p+1).
\end{equation}
For $p=1$, the interior points of
$\Delta_{p,q}$ are all on the line $x=0$. However, for $p>1$,
this is not the case as they are covered by the 
lines $x=i$ ($0 \le i \le p-1$). 
Since these lines are parallel, they cannot be
contained within any one line.
Incidentally, this collection of counterexamples
is related to the following summation identity.

\begin{proposition}
If $(p,q)\in \Z_+^2$ and $\gcd(q+1,p)=1$, then
\begin{equation} \label{countingidentityone}
 \sum_{i=0}^{p-1} \left\lceil q-\frac{i(q+1)}{p} \right\rceil 
= \frac{q}{2}(p+1).
\end{equation}
\end{proposition}

\begin{proof}
This is a simple counting argument using 
(\ref{pointsindeltapq}), the fact that
the lines $x=i$ ($0\le i \le p-1$) cover 
$\int(\Delta_{p,q})\cap \Z^2$, and the fact that there are
\[ \left\lceil q-\frac{i(q+1)}{p} \right\rceil  \]
lattice points on the line $x=i$ and $\int(\Delta_{p,q})$.
\end{proof}

The counterexample above can be generalized to 
a triangle with vertex set
\[ \{\ (-r,0),\ (0,q),\ (p,-1)\ \}, \]
where $p$, $q$, and $r$ are positive integers and
$\gcd(r,q) = \gcd(p,q+1) = 1$. The corresponding
identity below is similar to (\ref{countingidentityone}). 

\begin{proposition}
If $(p,q,r)\in \Z_+^3$ and $\gcd(q,r) = \gcd(q+1,p) = 1$, then
\[ \sum_{i=1}^{r-1} \left\lceil \frac{iq}{r} \right\rceil 
+ \sum_{i=0}^{p-1} \left\lceil q-\frac{i(q+1)}{p}\right\rceil
= \frac{q(p+r)+(r-1)}{2}. \]
\end{proposition}

Pick's theorem itself is a counterexample since only in $\R^2$ is
the volume of a lattice polytope $P$ an equality in terms of the number of
lattice points on $\partial P$
and in $\int(P)$. Reeve's tetrahedra
\cite{RV}
have vertices
\[ \{\ (0,0,0),\ (1,0,0),\ (0,1,0),\ (1,1,n)\ \}, \]
where $n\in \Z_+$.
These tetrahedra are clearly clean and empty, yet their volume increases
with $n$. Reeve concluded that the volume of lattice 
of a polyhedron $P$
cannot be expressed solely in terms of $b=|\partial P\cap \Z^3|$ and
$k=|\int(P)\cap \Z^3|$. However, Reeve was able to find and analogue
to Pick's theorem by considering sublattices.

Let $L_n^d$ denote the lattice consisting of all points 
in $\R^d$ whose coordinates are  multiples of $\frac{1}{n}$.
Note that $L_1^d = \Z^d$. For a lattice $d$-polytope $P$, let 
\[ b_n = b_n(P) = |\partial P \cap L_n^d|
\quad \text{and}\quad k_n = k_n(P) = |\int(P) \cap L_n^d|. \]
Reeve showed in \cite{RVnote} that 
\begin{equation} \label{reeveinrthree}
2n(n^2-1)\cdot \Vol(P) = b_n - n b_1 + 2(k_n - n k_n) +
(n-1)[2\chi(P) - \chi(\partial P)],
\end{equation}
where $n\ge 2$ and $\chi$ is the Euler characteristic.
Reeve also conjectured a similar formula for $d=4$.
Soon after, Macdonald \cite{IM} not only confirmed Reeve's conjecture, but
generalized Reeve's formulas to
\begin{equation} \label{macdonaldvolume}
\begin{split}
(d-1)d!\cdot \Vol(P)  = \sum_{i=1}^{d-1} (-1)^{i-1} \binom{d-1}{i-1}
(b_{d-i} + 2k_{d-i}) \\ + (-1)^{d-1}[2\chi(P) - \chi(\partial P)].
\end{split}
\end{equation}
In 1996, Kolodzieczyk \cite{KK96} showed that (\ref{macdonaldvolume}) has
the following alternate form in which only the numbers $k_n$ of interior lattice
points in the sublattice $L_n$ are required. 
\begin{equation} \label{kkequation}
d! \cdot \Vol(P) = \sum_{i=0}^{d-1} (-1)^i \binom{d}{i} k_{d-i} + (-1)^d [\chi(P)-\chi(\partial P)].
\end{equation}
For a $d$-polytope $P$ all of whose vertices join $d$ edges, 
$\chi(P)-\chi(\partial P) = (-1)^d$ and 
(\ref{kkequation}) assumes the form
\begin{equation}
d!\cdot \Vol(P) = \sum_{i=0}^{d-1} (-1)^i \binom{d}{i} k_{d-i} + 1. 
\end{equation}
Recently, Kolodzieczyk showed in \cite{KK00} that Pick-type formulas exist
even when considering only the points in $L_n\cap P$
whose coordinates are odd multiples of $\frac{1}{n}$.

Though no formula for the volume of a lattice polytope in terms of the
numer of its boundary and interior lattice points exist, 
there does exist
a Pick-type inequality with sharp bounds on the volume of a
polyhedron. Any polyhedron in $\R^3$ that does
not intersect itself has an associated planar graph. 
We can use elementary graph theory and Euler's formula to
obtain a lower bound on the volume of lattice polyhedra.

\begin{proposition} \label{picksinequality}
Let $P\subset \R^3$ be a convex lattice polyhedron with $b$ lattice
points on the boundary and $k\ge 1$ interior lattice points. Then
\[ \Vol(P) \ge \frac{2b+3k-7}{6}. \]
\end{proposition}

\begin{proof} By (\ref{volbytriangulation}), it
suffices to show $P$ has a lattice triangulation $\T$ satisfying
\[ |\T|\ge 2b+3k-7. \] Any refinement sequence starting with
any  basic triangulation of $P$ will suffice.
Consider the graph $G$ whose vertex set $\V(G)$ 
is $\partial P \cap \Z^3$
and whose edge set  consists of the edges of $\partial P$.
Since $P$ is convex, $G$
is planar. (Just embed $G$ into $\mathbb{S}^2$ and then map onto $\R^2$.)
Thus every triangulation of 
$G$ has $2\cdot |\V(G)| - 4$ faces, 
implying $\partial P$ can be triangulated into $2b-4$
triangles. Since $k\ge 1$, $P$ can be triangulated into $2b-4$ subtetrahedra
using any interior lattice point.
The remaining interior
points refine this triangulation, via Theorem \ref{refinementtheorem},
into a full triangulation having at least $2b-4 + 3(k-1)$
subtetrahedra.
\end{proof}

This bound is sharp (consider any clean, non-empty lattice tetrahedron).
A  similar result is proved (independently) in \cite[Lemma 3.5.5]{JDJRFL} by
De Loera, Rambau, and Santos Leal. While their proof does not 
assume that $P$ is nonempty, it does assume that $P\cap \Z^3$
is in general position. 
They also showed that any polyhedron with $n$ vertices
has a triangulation consisting of at most $2n-7$ tetrahedra. 

\section{Future Research}

What can we say about an upper bound for the volume of polyhedra?
Hensley proved in \cite{DH} that the volume of any lattice
$d$-polytope is bounded above by a function in terms of $d$
and the number of interior lattice points. 
Ziegler improved Hensley's results and
showed in \cite{LZ} that the volume of a $k$-point
lattice $d$-polytope $P$ satisifies
\[ \Vol(P) \le k\cdot [7(k+1)]^{d\cdot 2^{d+1}}. \]
In $\R^3$, a few simple computations seem to indicate that the upper
bound on the volume of clean tetrahedra is linear in $k$. Moreover,
maximal-volume tetrahedra also appear to have collinear interior lattice
points. Unlike the minimal-volume tetrahedra, however, the line
containing the interior points does not pass through any vertex.
Based on these computations, we make the following conjectures.

\begin{conjecture} \label{firstconjecture}
If $T\in \S_k^3$, then $\Vol(T) \le \frac{1}{6!}(12k+8)$.
\end{conjecture}

\begin{conjecture} \label{secondconjecture}
If $T\in \S_k^3$ and $\Vol(T) = \frac{1}{6!}(12k+8)$, then
$T\simeq T_{2k+1,4k+3,12k+8}$. If $k>1$, then
$\int(T)\cap \Z^3$ is a set of collinear points that does
not pass through any vertex of $T$.
\end{conjecture}

These conjectures are true for $k=1$, as proved independently
by Kasprzyk \cite{AK} and Reznick \cite{BR}. It may be possible
to gain some new insight on the relationship between
lattice points of a lattice $d$-polytope and its volume
using Ehrhart theory.
For example, it is well known that every convex lattice $d$-polytope
is associated with an Ehrhart polynomial of degree $d$. Moreover, the
leading coefficient of the associated polynomial is precisely the
volume of the $d$-polytope (cf. \cite{MBSR}, \cite{EE}).

\section{Acknowledgements}

I would like to thank my thesis advisor Bruce Reznick for having introduced
me to lattice polytopes. I am also grateful to Phil Griffith, former
director of graduate studies of the math department
 at the University of Illinois in
Champaign-Urbana, who was directly responsible for my opportunity
to work with Bruce Reznick. I also would like to thank Matthias Beck and
Sinai Robins, the authors of
\cite{MBSR}, for their enlightening course on lattice point enumeration
in Banff, Canada.

%%%%%%%%%%%%%%%%%%%%%%%%%%%%%%%%%%%%%%%%%%%%%%%%%%%%%%%%%%%%%%%%%%%%%%%%%%%%
%
% R E F E R E N C E S
%
%%%%%%%%%%%%%%%%%%%%%%%%%%%%%%%%%%%%%%%%%%%%%%%%%%%%%%%%%%%%%%%%%%%%%%%%%%%%

\bigskip

\end{document}